\documentclass[12pt]{amsart}
\pdfoutput=1

\usepackage{amsthm,amssymb,amsmath,amscd}
\usepackage{amsthm}
\usepackage{pgf, tikz}
\usepackage{graphicx,subcaption}
\usepackage{verbatim}

\theoremstyle{definition}
\newtheorem{definition}{Definition}

\theoremstyle{plain}
\newtheorem{theorem}[definition]{Theorem}
{Lemma}
{Theorem}
\newtheorem{lemma}[definition]{Lemma}

\newtheorem{fact}[definition]{Fact}
\newtheorem{claim}[definition]{Claim}
{Conjecture}

\theoremstyle{remark}

\newcommand{\hb}{\hat{b}}
\newcommand{\tH}{J}
\newcommand{\hH}{\hat{H}}
\newcommand{\hn}{\bar n}
\newcommand{\hG}{\hat G}
\newcommand{\hV}{\hat V}
\newcommand{\hE}{\hat E}

\newcommand{\shm}{\sqrt{\hat m}}

\title[Long even cycles in 3-coloured graphs]
{Long monochromatic even cycles\\ in 3-edge-coloured graphs\\
	of large minimum degree}

\author{Tomasz \L{}uczak}
\address{Adam Mickiewicz University\\
	Faculty of Mathematics and Computer Science\\
	ul.\ Uniwersytetu Pozna\'nskiego 4,
	61-614 Pozna\'n, Poland}
\email{\tt tomasz@amu.edu.pl}
\author{Zahra Rahimi}
\address{Department of Mathematical Sciences\\
	Isfahan University of Technology\\
	Isfahan, 84156-83111, Iran}
\email{\tt zahra.rahimi@math.iut.ac.ir}
\thanks{The first author partially supported by National Science Centre, Poland, grant
	2017/27/B/ST1/00873.}



\begin{document}
		
	\begin{abstract} We show that for every $\eta>0$, there exists $n_0$ such that for every even $n$, $n\ge n_0$, 
		and every graph 
		$G$ with $(2+\eta)n$ vertices and minimum degree at least $(7/4+4\eta)n$, each colouring of the edges of $G$ with 
		three colours results in a monochromatic cycle of length $n$. 
	\end{abstract}
	
	\maketitle
	
	\date{January 1st, 2020}
	
	\section{Introduction}
	
	We write $G\to (H)_r$ when for each colouring 
	of the edges of a graph~$G$ with $r$ colours 
	there exists a monochromatic copy of a graph $H$ in~$G$. 
	Ramsey Theorem states that for every $H$ there exists  $N$ such that  
	$K_N\to (H)_r$; the smallest $N$ with this property  
	we denote $R_r(H)$. Schelp~\cite{Schelp} observed that for sparse graphs $H$, 
	such as paths or cycles,  we may expect
	that $G\to (H)_r$ 
	for all $G$ which have  $R_r(H)$ vertices, provided the  minimum degree 
	of $G$ is large enough. 
	For  $H$ which is a path or a cycle and just two colours this  problem has been thoroughly 
	studied in a series of papers (see \cite{BLSSW, GS, LNS, Sch, W}), so now we know the minimum value of the constant  $c$ for which, 
	in the relation $K_N\to (C_n)_2$, the complete 
	graphs $K_N$ can be replaced by $G$ with  $\delta(G)\ge cN$, at least for large $n$.
	Not surprisingly, the optimal constant $c$ depends on the parity of $n$.  
	
	Schelp's problem for  cycles and more than two colours seems to be much harder.
	Let us mention first that the value of the Ramsey number $R_k(C_n)$ is known only 
	for some special cases (see \cite{BS, FL, FL2, JS, KSS,L}).
	Recently, the authors of this paper have proved the following asymptotic version for 
	Schelp's problem for three long odd cycles (see~\cite{LR}).
	
	\begin{theorem}\label{thm:mainold}
		For every $\eta>0$ there exists $n_0$ such that for every odd $n\ge n_0$ and every graph $G$ on 
		$(4+\eta)n$ vertices with $\delta(G)\ge (7/2+2\eta)n$ each  colouring of
		the edges of $G$ with three colours leads to a monochromatic cycle of length $n$. 
	\end{theorem}
	
	In this paper we show an analogous result for even cycles. 
	Let us recall that for $n$  even and large enough we have $R(C_{n},C_{n},C_{n})=2n$ 
	(see Benevides and Skokan~\cite{BS}). 
	On the other hand, in Section~\ref{sec:example} we give an example of a graph $G$ on $8\ell$ vertices and 
	$\delta(G)\ge 7\ell-2$ which has a 3-edge-colouring without monochromatic $C_{4\ell}$. 
	The following theorem is a strong indication that, if $\hn$ is a large enough even number,  each 3-edge-colouring of a graph $G$ 
	with $2\hn$ vertices and minimum degree at least $\lceil 7\hn/4\rceil$ contains a cycle of 
	length $\hn$. 	
	\begin{theorem}\label{thm:main0}
		For every $\eta>0$ there exists $\hn_0$ such that for every even $\hn\ge \hn_0$ and every graph $G$ on 
		$2(1+\eta)\hn$ vertices with $\delta(G)\ge (7/4+2\eta)\hn$ each  colouring of
		the edges of $G$ with three colours leads to a monochromatic cycle of length $\hn$. 
	\end{theorem}
	
	As was observed in \cite{L}, due to Szem\'eredi's Regularity Lemma, if we are interested only in 
	asymptotic results for large $n$, dealing with monochromatic even cycles are basically equivalent to studying matchings 
	contained in one monochromatic component (see \cite{FL2} and \cite{LR} for a precise statement of this correspondence and its proof). 
	In particular, in order to prove Theorem~\ref{thm:main}, it is enough to show the following result.  
	
	\begin{theorem}\label{thm:main}
		For every $\eta>0$ there exists $n_0$ such that for every even $n\ge n_0$ and every graph $G$ on 
		$(2+\eta)n$ vertices with $\delta(G)\ge (7/4+4\eta)n$ each  colouring of
		the edges of $G$ with three colours leads to  a monochromatic 
		component which contains a matching saturating $n$ vertices. 
	\end{theorem}

	\section{Basic tools}
	
	In this section we introduce notations 
	we use throughout the paper and give a number of elementary facts on matchings and monochromatic components of coloured graphs.
	
Instead of the minimum degree, here and below we rather use the notion of density.
We say that a subgraph $G$ of a graph $H$  is {\em $b$-dense in~$H$} if for each vertex $v$ 
the difference between the degree of   $v$ in $G$ and its degree in $H$ 
is less than~$b$. Moreover, if a graph $G$ on $N$  is $b$-dense in $K_N$, which means that $\delta(G)\ge N-b$,  
we often just say that $G$ is $b$-dense.  
	
We start with a few  elementary observations.

	\begin{fact}\label{l:ind}
		Let $G=(V,E)$ be a graph which contains no matchings saturating $n$ vertices and let 
		$v\in V$ be a vertex of $V$ of degree at least $n-1$. Then a graph $G'=(V,E')$ obtained from $G$ 
		by adding to it additional edges incident to $v$ contains no matchings saturating $n$ vertices either.
	\end{fact}
	
	\begin{proof}
		Let us suppose that $G'$ contains a matching $M'$ saturating $n$ vertices and $\{v,w\}\in M'$. 
		Then 	$v$ has a neighbour $u$ in $G$ not saturated by $M'\setminus \{v,w\}$, and so $M=M'\setminus \{v,w\}\cup \{v,u\}$ 
		is a matching in $G$ contradicting the assumption.
	\end{proof}
	
	\begin{lemma}\label{l:match1}
		Let $G=(V,E)$ be a $b$-dense subgraph of the complete bipartite graph with bipartition $V=V_1\cup V_2$, where 
		$|V_1|\le |V_2|$. Then the following holds.
		\begin{enumerate}
			\item [(i)] If $|V_2|>b $, then there exists  a component in $G$ with at least $|V_1|+|V_2|-2b$ vertices. 
			\item [(ii)] If $|V_2|>2b $, then  there exists a component in $G$ which contains all vertices of $V_1$ and at least  $|V_2|-b$ vertices of $V_2$.
			\item [(iii)] If $|V_1|>b$ and $|V_2|>2b $ then $G$ is connected and contains a matching saturating all vertices of~$V_1$. 
		\end{enumerate}	
	\end{lemma}         
	
	\begin{proof} If $|V_2|>b$ then $G$ contains at least one edge $\{v_1,v_2\}$, where $v_1\in V_1$, $v_2\in V_2$. Since $v_1$ has at least $|V_2|-b$ neighbours
		and $v_2$ at least $|V_1|-b$ neighbours, (i) follows. 
		
		If $|V_2|>2b$ then each pair of vertices from $V_1$ share a neighbour, so all of them are contained in one component. Since 
		each vertex from $V_1$ has at least  $|V_2|-b$ neighbours, we get (ii). 
		
		Finally, if $|V_1|>b$ and $|V_2|>2b$, then by (ii) all vertices of  $V_1$ belong to one component and each vertex from $V_2$ has a neighbour in $V_1$,
		so~$G$ is connected.  
		Suppose that it contains no 	matching saturating all vertices from $V_1$. 
		From Hall's condition we infer that for some 
		$S\subseteq V_1$ we have 
		\begin{equation}\label{e:2}
		|N(S)|<|S|.
		\end{equation}
		But each vertex from $V_1$ has at least  $|V_2|-b>b$ neighbours, so we must have $|S|>b$. This, in turn, implies that 
		each vertex $w\in V_2$ has a neighbour in $S$, and so 
		$$ |N(S)|=|V_2|\ge |V_1|\ge |S|,$$
		condtradicting (\ref{e:2}).  
	\end{proof}
	
	\begin{lemma}\label{l:match2}
		Let $G=(V,E)$ be a graph which is $N/8$-dense in $K_N$, 
		whose vertices are coloured with three colours,  
	and let $F_1=(V_1, E_1)$ and $F_2=(V_2, E_2)$  denote two monochromatic components of the first and the second colours respectively.  
		Then either there exists a component of the third colour which 
		contains a matching saturating $N/2$ vertices, or 
		$$\min\{|V_1\setminus V_2|,| V_2\setminus  V_1|\}<N/4.$$
		
		In particular, if $G$ contains no monochromatic component with matching saturating 
		at least $N/2$ vertices, then each two monochromatic components of at least $N/2$ vertices each 
		share more than $N/4$ vertices.
	\end{lemma}
	
	\begin{proof} This is a direct consequence of Lemma~\ref{l:match1}(iii).
	\end{proof}
	
	Our next result gives a simple structural characterization of graphs without large matchings.
	
	\begin{lemma}\label{tutte}
		If a graph $F=(V,E)$ on $N$ vertices 
		contains no matching saturating at least $n$ vertices, then
		there exists a partition $V=S\cup T\cup U$ of the vertex set of $F$ such that the following holds.
		\begin{enumerate}
			\item[(i)] The subgraph induced in $F$ by $T$ has maximum degree at most $\sqrt{N}-1$.
			\item[(ii)] There are no edges between the sets $T$ and $U$.
			\item [(iii)]
			\begin{equation*}\label{eq:tutte1}
			|U|+2|S|=|S|+|V|-|T|<n+\sqrt{N}.
			\end{equation*}
			In particular,
			\begin{equation*}\label{eq:tutte2}
			|T|+|U|/2>|V|-n/2-\sqrt{N}/2.
			\end{equation*}
			\item [(iv)] $|T|\le |V|-n/2$.
		\end{enumerate}
	\end{lemma}  
	
	\begin{proof} The fact that there exists a partition $V=S\cup T\cup U$ for which 
		(i)-(iii) hold is a well known consequence of Tutte's Theorem (see, for instance, \cite{LR}).
		If $|T|>|V|-n/2$, then  choose 
		$T'\subseteq T$ such that $|T'|=|V|-\lceil n/2\rceil$ and put $S'=V\setminus T'$ and $U'=\emptyset$. 
		The new partition $V=S'\cup T'\cup U'$ clearly fulfils (i)-(iv).
	\end{proof}
	
Our next, somewhat technical,  result  guarantees the existence of a large matching in a dense subgraph of a specific graph.  
	Let  $\hH=\hH(m_0;m_1,m_2,m_3)$ denote
	the graph whose vertex set is partitioned into disjoint 
	sets $Z_0, Z_1,Z_2, Z_{3}$, of sizes $m_0, m_1, m_2, m_3$, respectively,
	and the edges of $\hH$ are those joining 
	two different sets of this $4$-bipartition and those contained in the set $Z_0$.
	The following result is proved in~\cite{LR} as a direct consequence of Lemma~\ref{tutte}.
	
	\begin{lemma}\label{l:m2}
		Let $G$ be a $b$-dense subgraph of $\hH(m_0;m_1,m_2,m_3)$, $\hat m= m_0+m_1+m_2+m_3$, 
		and let 
		$m_1, m_2,  m_3\le 2b+2\sqrt{\hat m}$ and $b\ge 6\sqrt{\hat m}$.  
		Then $G$ contains a matching $M$ saturating at least $L-5\sqrt{\hat m}$ 
		vertices, where 
		\begin{align*}
		L=L(m_0;m_1,m_2,m_3)&\ge \min\{\hat m,m_0+\hat m-2b\}.
		\end{align*}
		
		If, furthermore, $m_0>0$ and $L\ge 2b+3\shm$, then $G$ is connected. \qed
	\end{lemma}

	\medskip
	
	\section{A large monochromatic component}\label{sec:large}
	
	We first study the case when the 3-coloured $G$ contains a  large monochromatic component with at least $ 3N/4$ vertices.
	The main result of this section is the following lemma. Here and below we assume that all results and inequalities are true only
	for $N$ which is large enough. 
	
	\begin{lemma}\label{l:giant}
		Let $G=(V,E)$ be a graph which is $b$-dense in $K_N$, where $N=(2+\eta)n$ and $b<(1/8-3\eta)N<n/4-\eta N$,  for some $\eta>0$.
		For each colouring of the edges of $G$ with three colours which results in  
		a monochromatic component $F_1=(V_1,E_1)$ on at least $3N/4$ vertices, there  exists
		a matching saturating at least $n$ vertices contained in 
		a monochromatic component in one of the colours.
	\end{lemma}
	
	Let us assume that none of monochromatic components of $G$ contains a matching saturating $n$ vertices.
	In particular, $F_1$ contains no such matching, so we can use Lemma~\ref{tutte} to find a partition 
	$V_1=S_1\cup T_1\cup U_1$ which fulfils all conditions (i)-(iv) listed in this result. 
	We will show that in the subgraph  of $G$ induced by 
		$$W=(V_1\setminus S_1) \cup (V\setminus V_1)=T_1\cup U_1\cup (V\setminus V_1)$$
 there exists a monochromatic component in one of the two other colours with 
	a matching saturating at least $n$ vertices.
	
	Let us remark first that 
	\begin{equation*}\label{e:t1}
	\begin{aligned}
	2|S_1|+|U_1|&\le n+\sqrt{N},\\	
	|W|=|V|-|S_1|&\ge N+|U_1|/2-n/2-\sqrt N/2\\
	&>3N/4+|U_1|/2,\\
	|T_1|&= |V|-|V\setminus V_1|-|S_1|-|U_1|>N/2-|V\setminus V_1|\\
	&>N/4,
	\end{aligned}
	\end{equation*}
	where here and below $b<N/8-\eta N$.  	
	Note that all edges between sets $T_1$, $U_1$, and  $V\setminus V_1$ are either of the second or the third colour, 
	and the subgraph spanned in $T_1$ by the edges of the first colour has maximum degree at most $\sqrt{N}$.
	Thus, if we set $m_0=|T_1|$, $m_1=|U_1|$, and $m_2=|V\setminus V_1|$ and denote the graph spanned by the edges of the second and the third colour
	contained in $W$ by $\tH$, then,  
	in order to show Lemma~\ref{l:giant}  it is enough to  prove the following result. 
	
	\begin{lemma}\label{l:m2a}
		Let $\tH$ be a $\hb$-dense subgraph of $\hH=\hH(m_0;m_1,m_2, m_3)$ such that $m_3=0$, $b'=b+\sqrt N\le N/8-6\sqrt{N}$, and 
		\begin{enumerate}
			\item[(i)] $m_0+m_1+m_2\le N$,
			\item[(ii)] $m_0+m_1/2+m_2\ge 3N/4$,
			\item[(iii)] $m_2\le  N/4$,
			\item[(iv)] $m_0> N/2-m_2\ge N/4$.
		\end{enumerate}
		Then each 2-colouring of edges of $\tH$ leads to a monochromatic component which contains a matching saturating at least $N-3\sqrt N$ vertices.
	\end{lemma}
	
	\begin{proof}
		In what follows we denote by $\tH_i=G_i[W]$, for $i=2,3$, the graphs induced in $\tH$ by the second and the third colours, respectively.
		Let us start with a number of observations. 
		
		\begin{claim}\label{cl11}
			Suppose that in either the second or the third  colour, say the second one, there exists a union of  components $F'_2=(V'_2,E'_2)$ such that 
			$|W\setminus V'_2|\ge N/4$ and $V'_2\cap T_1 \neq \emptyset$. Then all vertices of $V'_2\cap T_1$ are contained in the same 
			component $F_3=(V_3,E_3)$ of the third colour such that $|T_1\setminus  V_3|\le N/8$. 
		\end{claim}
		
		\begin{proof}
			Note that each pair of vertices  $v,w \in V'_2\cap T_1$ must share a neighbour of
			the third colour  in  $W\setminus V'_2$, and so the set  $V'_2\cap T_1$ is contained in the same component $F_3=(V_3,E_3)$ of $\tH_3$. 
			Let $v\in V_3\cap T_1$. Since every neighbour of $v$ in $T_1$ belongs to $V_3$, $F_3$ covers all but  at  most  $b'<N/8$ vertices of $T_1$.
		\end{proof}  
		
		\begin{claim}\label{cl2} The set  $W$ of vertices of $\tH$ can be covered by two monochromatic 
			components of different colours. Furthermore, either $T_1$ 
			is contained in one of them, or each of them   
			covers all but at most $b'$ vertices of $W$. 
		\end{claim}
		
		\begin{proof}
			Let us consider first the case when one monochromatic component, say $F_2=(V_2,E_2)$, contains all vertices of $T_1$. 
			Since $\tH$ is $b'$-dense, (iv) implies that each pair of vertices from $W\setminus V_2$ has in $T_1$ a common neighbour. Since all 
			edges joining $W\setminus V_2$ with $T_1$ are of the third colour, all vertices of   $W\setminus V_2$  belongs to 
			one monochromatic component $F_3=(V_3,E_3)$ of the third colour. 	
			
			Now let us suppose that there are several components of $\tH_2$ which intersect $T_1$. Then there exists a component
			$F'_2=(V'_2,E'_2)$ in $\tH_2$ which contains at most $|W|/2$ vertices. Using Claim~\ref{cl11}  we infer that
			there exists a component $F_3=(V_3,E_3)$ 
			of the third colour, such that $V'_2\cap T_1\subseteq V_3\cap T_1$ and $|T_1\setminus  V_3|\le N/8$. The case  $T_1\subseteq V_3$ has been treated above so 
			we assume that  $T_1\setminus V_3\neq \emptyset$.  
			Note also that if $|W\setminus V_3|\ge N/4$, then, again by Claim~\ref{cl11}, there exists 
			a component $F_2=(V_2,E_2)$ in $\tH_2$  such that $V_3\cap T_1\subseteq V_2\cap T_1$ and so 
			$$|V_2\cap T_1|\ge |V_3\cap T_1|\ge |T_1|-N/8\ge N/8>b'.$$ 
But this means that every vertex $v\in T_1\setminus V_3$ has in $V_3\cap T_1\subseteq V_2$ a neighbour of the second colour. Consequently, $T_1\subseteq V_2$ and, again, 
			this case is reduced to the one we have already considered. 
			
			Thus, it is enough to study the case when  $|W\setminus V_3|< N/4$. Each pair of vertices 
			$v,w\in T_1\setminus V_3$ has a common neighbour in $V_3$, so all such vertices lie in one  
			component  $F_2=(V_2,E_2)$ of the second colour and so  $T_1\subseteq V_2\cup V_3$. 
			Since all neighbours of a vertex $v\in T_1\setminus V_2$ in $T_1$ belong to $F_3$,  
			we have  $|T_1\setminus V_2|<N/8$.
			Observe  that there are no edges between $V_2\setminus V_3$ and 
			$V_3\setminus V_2$. Thus, considering vertices $v\in (V_2\cap T_1)\setminus V_3$, 
			and $w\in (V_3\cap T_1)\setminus V_2$, which both have in $\tH$ degree at least $|W|-b'$, we infer that 
			\begin{equation*}\label{eqqq}
			|V_2\setminus V_3|, |V_3\setminus V_2|\le b'<N/8.
			\end{equation*}
			
			In order to complete the proof  we need to show that components $F_2=(V_2,E_2)$ and $F_3=(V_3,E_3)$ cover not only $T_1$ but the whole set~$W$.
			
			Suppose that it is not the case and 
			$|W\setminus (V_2\cup V_3)|>0 $.
			From the fact that  $|W|=m_0+m_1+m_2\ge m_0+m_1/2+m_2\ge 3N/4$, we get 
			\begin{equation*}\label{eq555}
			\begin{aligned}
			|W\setminus (V_2\cup V_3)|+|V_2\cap V_3|&= |W|- |V_2\setminus V_3|-|V_3\setminus V_2|\\
			&\ge |W|-2b'> N/2 .
			\end{aligned}
			\end{equation*}
			There are no edges between $V_2\cap V_3$ and $W\setminus (V_2\cup V_3)$. Moreover, since 
			$$|V_2\cap V_3\cap T_1|\ge |T_1|-	|V_2\setminus V_3|-|V_3\setminus V_2|>0,$$ we have  
			$|W\setminus (V_2\cup V_3)|<b'$. 
			Consequently,  
\begin{equation}\label{eq55}
|V_2\cap V_3|\ge N/2-b'.
\end{equation}
			
			We claim that $W\setminus (V_2\cup V_3)\subseteq U_1$.
			Let us suppose that this is not the case and there exists 
			$v\in (V\setminus V_1)\setminus (V_2\cup V_3)$. Then 
			\begin{equation*}\label{eq6a}
			|(T_1\cup U_1)\cap V_2\cap V_3|\le b' , 
			\end{equation*}
			and so, by \eqref{eq55}, 
			\begin{align*}\label{eq6b}
			| (V\setminus V_1)\cap V_2\cap V_3|&= |V_2\cap V_3|-|(T_1\cup U_1)\cap V_2\cap V_3|\\
			&>N/2-2b'	> N/4, 
			\end{align*}
			while the condition (iii)  of Lemma~\ref{l:m2a} states that $m_2=|V\setminus V_1|\le N/4$. 
			
Thus $W\setminus (V_2\cup V_3)\subseteq U_1$. 
Hence, if $W\setminus (V_2\cup V_3)\neq \emptyset$, then
	\begin{equation*}\label{eq6}
			|(T_1\cup (V\setminus V_1))\cap (V_2\cap V_3)|\le b' , 
		\end{equation*}
and 
\begin{align*}
	|T_1\cup (V\setminus V_1)|&\le |(T_1\cup (V\setminus V_1))\cap (V_2\cap V_3)|+|V_2\setminus V_3|+|V_3\setminus V_2|\\ 
	&\le 3b'< 3N/8. 
\end{align*}			
while the condition (iv) of Lemma~\ref{l:m2a} states that
$$|T_1\cup (V\setminus V_1)|=m_0+m_2>N/2.$$
			This final contradiction shows that $W=V_2\cup V_3$.
		\end{proof}
		
		From Claim~\ref{cl2} we know that $W$ can be covered by two monochromatic components 
		$F_2=(V_2,E_2)$ and  $F_3=(V_3,E_3)$  of different colours. 
		We may and shall  assume that $|V_2\cap T_1| \ge |V_3\cap T_1|$. 
		Note also that by Claim~\ref{cl2} each vertex from $T_1\setminus V_2$ has at least $|V_2|-b'>|W|-2b'>N/2$  neighbours of the third colour in $V_2$. 
		Since all the vertices of $T_1\setminus V_2$ belong to $V_3$, due to Lemma~\ref{l:ind}, we can recolour
		all edges from $T_1\setminus V_2$ to $V_3$ by the third colour, not changing the size of the largest matching contained in $F_3$, provided it is smaller than 
		$n<N/2$. Thus, from now on, we assume that all such edges are of the third colour. 
		We  recall  that we have also claimed that there are no edges of the first 
		colour contained in $T_1$ from $G$, compensating this fact by replacing $b$ by $b'=b+\sqrt{N}$. 
		
		Since we have assumed that no monochromatic component contains a matching saturating at least $n$ vertices, using 
		Lemma~\ref{tutte} we can partition the set of vertices of $F_2=(V_2,E_2)$ into three sets $S_2\cup T_2\cup U_2$ for which the conditions 
		(i)-(iv) mentioned in the Lemma~\ref{tutte} hold. In particular, we have $|S_2|\le n/2+\sqrt{N}/2<N/4$. We show that 
		then in the graph $G_3[V\setminus (S_1\cup S_2)]$ there exists a component with a matching saturating at least $n$ vertices.
		Moreover, we shall assume that $G$ contains no edges of the second colour contained in $T_2$; if this is the case 
		we modify a graph and replace $b'$ by $b''= b'+\sqrt{N}<N/8-4\sqrt N$.
		
We define  $A_3=V\setminus (V_1\cup V_2)$,
		$Z_0=T_1\cap (T_2\cup (V_3\setminus V_2))$, $Z_1=T_1\cap U_2$, $Z_2=T_2\cap U_1$,
		$Z'_2=Z_2\cup (U_1\setminus V_2)$, $Z_3=T_2\setminus V_1$ and $Z'_3=Z_3\cup A_3$. 
		Let also $|Z_i|=n_i$, $i=0,1,2,3$, and $|Z'_i|=n'_i$,  $i= 2,3$.
		We shall show that the subgraph induced in $G_3$ by $Z_0\cup Z_1\cup Z_2\cup Z_3$, or a similar one, 
		with $Z_2$ replaced by $Z'_2$ and/or $Z_3$ 
		replaced by $Z'_3$, fulfils the hypothesis of Lemma~\ref{l:m2}, and so contains a component with a matching saturating at least $n$ 
		vertices, which contradicts our assumption.
		
		Let us start with a few observations.
		
		\begin{claim}\label{claim5}
			$|Z_0|=|T_1\cap (T_2\cup (V_3\setminus V_2))|< N/4$.
		\end{claim}  
		
		\begin{proof} Suppose that $|Z_0|\ge N/4$. Note that 
			all edges inside $Z_0$ and between $Z_0$ and $V\setminus (S_1\cup S_2)$
			are of the third colour.  
			Choose a subset   $T\subseteq Z_0$ with $N/4$ vertices. Since
			$|S_1\cup S_2|\le N/2$,  we have
			$$|V\setminus (S_1\cup S_2)|-|T |\ge N/4.$$ 
			So, by Lemma~\ref{l:match1}, the graph contains a matching contained in $F_3$ 
			saturating $N/2$ vertices which contradicts our assumption. 
		\end{proof}
		
		\begin{claim}\label{claim2}
			$|Z_1 |=|T_1\cap U_2|< N/4$.
		\end{claim}  
		
		\begin{proof}
			Suppose that $|Z_1|\ge N/4$. 
			Consider the set 
			$$W_1=T_2\cup (V_1\setminus (V_2\cup S_{1}))\cup A_3.$$ 
			From Lemma~\ref{tutte}(iii) we get   
			$|T_2|\ge |V_2|+|S_2|-N/2$, $|S_1|< N/4$.  Thus, since $|V_2|  +|(V_1\setminus V_2)|+|A_3|=|V|=N$,  we arrive at 
			\begin{align*}
			|W_1|&\ge|T_2|+|(V_1\setminus V_2)|-|S_1|+|A_3|\\
			&> |V_2| +|S_2|-N/2+|(V_1\setminus V_2)| -|S_1|+|A_3|\ge N/4. 
			\end{align*}
			Note that all edges between $Z_1$ and $W_1$ are of the third colour. Thus, 
			from Lemma~\ref{l:match1}, $F_3$ contains  a matching saturating $n$ vertices, contradicting our assumption.	
		\end{proof}
				
		\begin{claim}\label{claim3}
			$|Z_2|=|T_2\cap U_1|\le |Z'_2 |=|Z_2\cup(U_1\setminus V_2)|\le N/4.$\\
		\end{claim}  
		
		\begin{proof}
			Let 
			$$W_2=T_1\cup (V_2\setminus (V_1\cup S_{2}))\cup A_3.$$ 
			As in the previous claim, we   can argue that all edges between 
			$Z_2$ and $W_2$ are of the third colour, and
			\begin{align*}
			|W_2|&=|T_1\cup(V_2\setminus V_1)\setminus S_2|+|A_3|\\
			&\ge|T_1|+|V_2\setminus V_1|-|S_2|+|A_3|\\
			&\ge |V_1| +|S_1|-N/2+|(V_2\setminus V_1)| -|S_2|+|A_3|\ge N/4.
			\end{align*}
			Let $X\subseteq (T_1\cup (V_2\setminus V_1))\setminus S_{2}$ be such that $|X|+|A_3|=N/4$. 
			From Lemma~\ref{l:match1} we infer that $F_3$ contains  a matching $M$  between $X\cup A_3$ and $Z_2$ saturating
			$\min\{N/2, 2|Z_2|\}$ vertices. 
			
			We supplement $M$ by another matching contained in $F_3$. To this end consider a set $W_3=V\setminus (U_1\cup S_1 \cup A_3\cup X)$ and note that 
			all edges between $W_3$ and $U_1\setminus V_2$  are of the third colour.  
			Moreover
			$$|W_3|\ge |V|-|U_1|-|S_1|-|X|-|A_3|> N-N/2-N/4=N/4,$$
			so, by Lemma~\ref{l:match1}(ii),  there is a matching $M'$ in $F_3$ between $U_1\cap (V_1\setminus V_2)$ and $W_3$  
			vertex-disjoint with $M$,   saturating $\min\{N/2, 2|U_1\setminus V_2)|\}$ vertices.
			Thus,  $F_3$ contains a matching $M\cup M'$ saturating $2|Z'_2|$ 	vertices, and so $|Z'_2|< N/4$.
		\end{proof}    
		
		\begin{claim}\label{claim4}
			$|Z_3 |=|T_2\setminus V_1|\le |Z'_3 |=|Z_3\cup A_3|\le N/4$.
		\end{claim}  
		\begin{proof}
			It is enough to note that  $Z_3\subseteq Z'_3\subseteq V\setminus V_1$ and $| V\setminus V_1|\le N/4$.
		\end{proof}

		Before we proceed further we study the size of the set $V_3\setminus (V_2\cup S_1)$. 
		Let us suppose that  $|V_3\setminus (V_2\cup S_1)|\ge N/4$. From Claim~\ref{cl2} it can happen only when $T_1\subseteq V_2$. 
		But then all edges between $V_3\setminus (V_2\cup S_1)$ and $T_1$ are of the third colour, and, since $|T_1|=m_0>N/4$, by Lemma~~\ref{l:match1}, 
		there exists a matching in $F_3$ saturating $N/2$ vertices. 
		Note also that, by Claim~\ref{cl2}, $T_1\cup U_1\subseteq V_2\cup V_3$. Hence, from now on we assume that 
		\begin{equation}\label{eq201}
		|(T_1\cup U_1)\setminus V_2|+|A_3|\le |V_3\setminus (V_2\cup S_1)|<N/4.
		\end{equation} 
		
		Let us also observe that
		\begin{equation}\label{eq:last}
		\begin{aligned}
		n_0&+(n_0+n_1+n_2+n'_3) 
		\\
		&=n_0+|T_1|+|U_1|+|V\setminus V_2|-|U_1\setminus V_2| 
		\\ &\hspace{5truecm} 
		-|S_2|-(|U_2|-|U_2\cap T_1|) 
		\\
		&=n_0+m_0+m_1+m_2 -|U_1\setminus V_2|+|S_2| 
		\\ &\hspace{5truecm} 
		-(2|S_2|+|U_2|)+|U_2\cap T_1|
		\\
		&\ge n_0+|S_2\cap T_1|+|U_2\cap T_1|+m_0+m_1
		\\ &\hspace{5truecm} 
		+m_2-|U_1\setminus V_2| -N/2
		\\
		&\ge m_0+m_0+m_1+m_2 -|U_1\setminus V_2|-N/2
		\\
		&= 2(m_0+m_1/2+m_2)-m_2  -|U_1\setminus V_2|-N/2
		\\
		& \ge 3N/4-|U_1\setminus V_2|.
		\end{aligned}
		\end{equation}
		and, since $n_0=|Z_0|\le N/4$, 
		\begin{equation}\label{eq:last2}
		\begin{aligned}
		n_0+n_1+n_2+n'_3 &\ge 2n_0+n_1+n_2+n'_3-N/4
		\\ & \ge N/2-|U_1\setminus V_2|.
		\end{aligned}
		\end{equation}
		We remark here that from \eqref{eq:last}  and Claims~\ref{claim5}, \ref{claim2}, and~\ref{claim3}, 
		it follows that 
		\begin{equation}\label{eq201b}
		n_0\ge \big(3N/4-(n_1+(n_2+|U_1\setminus V_2|)+n'_3)\big) >0.
		\end{equation}		
		
		Now we split our argument into two cases.
		
		\medskip
		
		{\em Case 1.} $|S_1|\ge N/8$.
		
		\smallskip
		
		Since  $2|S_1|+|U_1|<N/2$, we have $|S_1|+|U_1|<3N/8$ and $|U_1|\le N/4$. All edges between $U_1\setminus V_2$ and $V\setminus (S_1\cup U_1\cup A_3\cup Z_0)$ are of
		the  third colour. Moreover, Claim~\ref{claim5} and \eqref{eq201} give
		\begin{equation}\label{eq202}
		\begin{aligned}
		|V\setminus (S_1\cup U_1\cup A_3\cup Z_0)|&\ge N-|S_1|-|U_1|-|A_3|-|Z_0|\nonumber\\
		& > 5N/8-(N/4-|U_1\setminus V_2|)-|Z_0|\nonumber\\
		&\ge |U_1\setminus V_2|+N/8.
		\end{aligned}
		\end{equation}
		Thus, by Hall's Theorem,  there is a matching $M'$ in the third colour saturating all vertices of $U_1\setminus V_2$. 
		Note that $M'$  is contained in $F_3$, since $U_1\setminus V_2 \subseteq F_3$,
		and saturates no vertices from $Z_0$. We supplement it by another matching $M''$ 
		contained in $F_3$ whose size we estimate applying Lemma~\ref{l:m2} to the sets 
		$Z_0$, $Z_1\setminus (U_1\setminus V_2)$, $Z_2\setminus (U_1\setminus V_2)$ and $Z'_3\setminus (U_1\setminus V_2)$. 
		Note that $M'$ saturates  no vertices from the above four sets.
		Moreover, by (\ref{eq:last}),
		\begin{align*}
		2n_0+n_1+n_2+n'_3 -|U_1\setminus V_2|\ge 3N/4-2|U_1\setminus V_2|,
		\end{align*}
		while (\ref{eq:last2}) gives
		\begin{align*}
		n_0+n_1+n_2+n'_3 -|U_1\setminus V_2| \ge  N/2-2|U_1\setminus V_2|.
		\end{align*}
		Since by \eqref{eq201b} $n_0>0$, from Lemma~\ref{l:m2} it follows that $G_3$ contains a matching 
		$M''$, contained in one component, which  saturates at least 
		$$N/2-2|U_1\setminus V_2|-5\sqrt N >n-2|U_1\setminus V_2|$$ vertices. 
		Moreover 
		$$n_0+n_1+n_2+n'_3\ge N/2-|U_1\setminus V_2|>N/4>2b''+3\sqrt N,$$
		so, by Lemma~\ref{l:m2}, $M''$ is contained in $F_3$.
		Consequently, $M'\cup M''$ is a matching saturating at least $n$ vertices.
				
		\medskip
		
		{\em Case 2.} $|S_1|\le N/8$.
		
		\smallskip

		For such $S_1$ we have 
		$$|V_2\setminus V_1|= |V\setminus V_1|-|A_3|\le N/4-|A_3|,$$
		and, by \eqref{eq201},  
		$$|(V_1\setminus V_2)|\le |S_1|+|(T_1\cup U_1)\setminus V_2|\le |S_1|+ N/4-|A_3|\le 3N/8-|A_3|.$$
		All of the edges between $A_3$ and the set $Y=V\setminus ((V_1\setminus V_2)\cup (V_2\setminus V_1) \cup A_3 \cup Z_0)$ are of the third colour and
		\begin{align*}
		|Y|
		&\ge |V|-(3N/8-|A_3|)-(N/4-|A_3|)-|A_3|-N/4\\
		&=N/8+|A_3|.
		\end{align*} 
		
		Thus, by Hall's Theorem, in $G_3$ one can find a matching $M'$, saturating all vertices of $A_3$. 
		By Lemma~\ref{l:match1}, this matching is contained in a component of third colour, covering all vertices of $A_3$, and 
		so it is contained in $F_3$.
		Just as in the  previous case, we supplement it by another matching $M''$ 
		contained in $F_3$ by applying Lemma~\ref{l:m2} to the sets 
		$Z_0$, $Z_1\setminus M'$, $Z'_2\setminus M'$ and $Z_3\setminus M'$. 
		Indeed, (\ref{eq:last}) gives
		\begin{align*}
		2n_0+n_1+n'_2+n_3-|A_3| \ge 3N/4-2|A_3|,
		\end{align*}
		and by (\ref{eq:last2}) we get
		\begin{align*}
		n_0+n_1+n'_2+n_3-|A_3|  \ge  N/2-2|A_3|.
		\end{align*}
		Since $n_0>0$, from Lemma~\ref{l:m2} it follows that there exists a matching  
		$M''$ in $G_3$ saturating at least 
		$$N/2-2|A_3|-5\sqrt N\ge n-2|A_3| $$ vertices. 
		$M''$ is contained in the subgraph induced in $G_3$ by $Z_0\cup Z_1\cup Z'_2\cup Z_3$ and, since 
		$$n_0+n_1+n'_2+n_3\ge N/2-|A_3|>N/4>2b''+3\sqrt N,$$
		by Lemma~\ref{l:m2}, this subgraph is connected and is contained in $F_3$.
		Consequently, $F_3$ contains a  matching saturating at least $n$ vertices.
		%
		%
		%
	\end{proof}
	
	\medskip
	
	\section{Monochromatic components of medium size}\label{sec:comp}
	
	Because of Lemma~\ref{l:giant}, in order to show Theorem~\ref{thm:main}, it is enough to study the case when each monochromatic component  has fewer than $3N/4$ vertices.  
	Let us start with the following fact from~\cite{LR}.
	
	\begin{lemma}\label{l:comp11}
		Let $G=(V,E)$ be a graph which is $N/8$-dense in $K_N$. Then  each  colouring of
		the edges of $G$ with three colours leads to  a monochromatic component 
		of size at least $N/2$. \qed
	\end{lemma}
	
	Clearly, a 3-coloured graph may have only one  monochromatic component with at least $N/2$ vertices. However, it turns out that if the graph
	has large enough minimum degree and the largest monochromatic component is not too big, then $G$ contains  monochromatic component of at least $N/2$ vertices in two different colours. Let us recall that $G_1=(V,E_1)$, $G_2=(V,E_2)$, and $G_3=(V,E_3)$, denote the subgraphs of 
	$G$ induced by edges of the first, the second, and the third colour, respectively.
	
	\begin{lemma}\label{l:comp}
		Let $G=(V,E)$ be a graph which is $N/8$-dense in $K_N$.  
		whose edges are coloured with three colours.
		If the largest monochromatic component 
		in  $G$ has fewer than $3N/4$ vertices, then 
		there exists another monochromatic component
		of at least $N/2$ vertices.
	\end{lemma}
	
	\begin{proof}
		Let $F_1=(V_1,E_1)$ be the largest monochromatic component of~$G$, 
		say, in the first colour, and let $|V_1|<3N/4$. Note  that by Lemma~\ref{l:comp11} we know that $|V_1|\ge N/2$.
		Let $F_2=(V_2, E_2)$ denote the largest monochromatic component in two  other colours.
		We assume that  $F_2$ is of the second colour and that $|V_2|<N/2$.
		Let $A=V_1\setminus V_2$, $B=V_2\setminus V_1$ and  $C=V\setminus (V_1\cup V_2)$ and $D=V_1\cap V_2$. 
		We consider two cases.
		
		\medskip
		
		{\em Case 1.} $|B|\ge N/8$.
		
		\smallskip 
		
		If  $|A\cup B|\ge N/2$, then, since  $|A|\ge |B|$, we have 
		$|A|\ge N/4$. Thus,  by Lemma~\ref{l:match1}(iii), $G_3[A\cup B]$ is connected graph with at least $N/2$ vertices, contradicting the choice of $F_2$. 
		
		If $|A\cup B|< N/2$ then $|C\cup D|\ge N/2$. We will argue that then both $C$ and $D$ have more than $N/8$ vertices each
		and so,  by Lemma~\ref{l:match1}(iii), $G_3[C\cup D]$ is 
		connected and contains at least $N/2$  vertices. 
		
		Note that  $|A\cup B|< N/2$ and $|B|\ge N/8$, so $|A|=|V_1\setminus V_2|<3N/8$.  
		Thus, since $|V_1|\ge N/2$, we have $|D|=|V_1\cap V_2|> N/8$. Moreover, 
		$|V_2|<N/2$ and so
		$$|C|\ge |C\cup A|-|A|=|V\setminus V_2|-|A|>N/2-3N/8\ge N/8. $$   
		
		\medskip
		
		{\em Case 2.} $|B|< N/8$.
		
		\smallskip 
		
		Note that for such small $B$ we have
		$$|C|= |V|-|V_1|-|B|>N/8>|B|.$$
		
		The remaining part of the argument we split  into  two subcases.	
		
		\medskip
		
		{\em Subcase 2a.} $|A\cup B|\ge N/2$.
		
		\smallskip 
		
		Then $|A|>3N/8$ and, by Lemma~\ref{l:match1}(ii), there exists a connected subgraph in $G_3$  which covers all $A\cup B$ 
		except  at most $N/8$ vertices of $A$. 
		Let $A'\subseteq A$ denote the set of vertices of $A$ which belong to this subgraph. 
		Note also that if  $|D|\ge N/4$, then, by Lemma~\ref{l:match1}(iii), $G_3[C\cup D]$ is connected  and has 
		$$|C|+|D|>|B|+|D|=|V_2|$$
		vertices contradicting our choice of $F_2$. Thus, 
		$|D|<N/4$ and 
		\begin{equation*}\label{11}
		|V_2|=|B|+|D|<3N/8.
		\end{equation*}
		Next we observe that each vertex of $C$ has only neighbours of the second colour in $A'$. Indeed, by Lemma~\ref{l:match1}(i), 
		a vertex $c\in C$ in $G_3[C\cup D]$, belongs to a component of $G_3$  which covers all but at most $N/4$ vertices of $C\cup D$.
		An edge of the third colour joining  $c$ to $A'$ would result in a component  in $G_3$ which covers all  but at most $N/8$ vertices of $A\cup B$ and all but $N/4$ vertices of  $C\cup D$. Such a component would have at least  $5N/8>3N/8>|V_2|$ vertices contradicting the choice of $F_2$. 
		Consequently, all edges between $A'$ and $C$ are of the second colour and, since $|C|>N/8$ and $|A'|\ge N/4$, by Lemma~\ref{l:match1}(iii),  $G_2[A'\cup C]$ is connected with at least $3N/8>|V_2|$ vertices, which, again is impossible. This completes the proof of this subcase.
		
		\medskip
		
		{\em Subcase 2b.} $|A\cup B|< N/2$.
		
		\smallskip 
		
		Since $|A\cup B|<N/2$, we have $|C\cup D|>N/2$. Consequently, if 
		$|D|\ge N/8$, then, by Lemma~\ref{l:match1}(iii),   $G_3[C\cup D]$ is connected and we are done. 
		
		Thus, let us assume that 
		$|D|<N/8$ and so   $|V_2|=|B|+|D|<N/4$, which, in turn, implies that  $|A\cup C|\ge 3N/4$ and $|C|>3N/8$. Note also that, since $|V_1|=|A|+|D|\ge N/2$, we have
		$|A|>3N/8$.  Hence, by  Lemma~\ref{l:match1}(ii), there exists a connected subgraph $F'_3=(V'_3,E'_3)$ of $G_3$ 
		which covers all $A\cup B$ except  at most $N/8$ vertices of $A$, 
		and there exists a connected subgraph $F''_3=(V''_3,E''_3)$ of $G_3$ which covers all $C\cup D$ except  at most $N/8$ vertices of $C$. 
		Let $A'=V'_3\cap A\subseteq A$ and  $C'=C\cap V''_3$.  
		If a vertex from $A'$ has a neighbour of the third colour in $C'$, then $G_3$  would contain a component which covers all vertices except  at most $N/8$ vertices of $A$ and  at most $N/8$ vertices of $C$. Such a  component would contain at least $3N/4$ vertices contradicting the choice of $F_2$.
		Thus, all edges between  $A'$ and $C'$ are of the second colour. Since $|A'|,|C'|>N/4$, by  Lemma~\ref{l:match1}(iii), $G_2[A'\cup C']$ 
		is connected and has more than $|V_2|$ vertices  contradicting the choice of $F_2$.  
	\end{proof}
	
	The following observation is a direct   
	consequence of Lemma~\ref{l:match2}. 
		
	\begin{lemma}\label{l:comp3}
		Let $G=(V,E)$ be a graph which is $N/8$-dense in $K_N$  
		whose vertices are coloured with three colours in such a way that it contains no monochromatic 
		component with a matching saturating at least $N/2$ vertices. Moreover, let us also assume that 
		all monochromatic components of $G$ are smaller than $3N/4$. Then, for each pair of monochromatic components 
		$F_1=(V_1,E_1)$,  $F_2=(V_2,E_2)$ of the first and the second colour respectively, such that $|V_1|, |V_2|\ge N/2$, 
		there exists a monochromatic component $F_3=(V_3,E_3)$ of the third colour such that $V=V_1\cup V_2\cup V_3$. 
	\end{lemma}
	
	\begin{proof}
		Since $G$ contains 
		no matchings saturating $N/2$ vertices contained in one monochromatic component, 
		from Lemma~\ref{l:match2} we get
		$|V_1\cap V_2|\ge N/4$. But then each pair of vertices from $V\setminus (V_1\cup V_2)$ has a common neighbour of the third colour in 
		$V_1\cap V_2$. Consequently, all vertices of $V\setminus (V_1\cup V_2)$ belong to one monochromatic component $F_3=(V_3,E_3)$ of the third colour.
	\end{proof}

	\section{Proof of Theorem~\ref{thm:main}}\label{proof}
	
	Since Lemma~\ref{l:giant} takes care of the case when $G$ contains a large monochromatic component of at  least $3N/4$ vertices, 
	in order to prove Theorem~\ref{thm:main}  it is enough to show  the following result.

	\begin{lemma}\label{l:small}
		Let $G=(V,E)$ be a graph which is $N/8$-dense in $K_N$  
		whose vertices are coloured with three colours in such a way that  no monochromatic 
		component is larger than $3N/4$. 
		Then $G$ contains a monochromatic component with a matching saturating at least $N/2-4\sqrt N$ vertices. 
	\end{lemma}
	
	\begin{proof}
		We shall show Lemma~\ref{l:small} by contradiction. Thus, we assume that 
		the edges of a graph $G=(V,E)$, $|V|=N=(2+\eta)n$, which is 
		$b$-dense in $K_N$ with $b<(1/8-\eta)N<N/8-10\sqrt N$, 
		are 3-coloured in such a way 
		that no  monochromatic component contains a matching saturating at least $N/2-4\sqrt N$ vertices. 
		The subgraphs of $G$ induced by the first, the second, and the third colour we denote by $G_1$, $G_2$, and $G_3$ respectively.
		Let $F_1=(V_1,E_1)$ be one of the largest monochromatic component (we assume that it is of the first colour) and 
		$F_2=(V_2,E_2)$ be the one of the second largest monochromatic component (we assume that it is of the second colour). From Lemma~\ref{l:comp} we have 
		$|V_1|\ge |V_2|\ge N/2$. Moreover, by $F_3=(V_3,E_3)$ we denote a monochromatic component of the third colour such that $V=V_1\cup V_2\cup V_3$
		whose existence is assured by Lemma~\ref{l:comp3}. 
		Note that since none of the components $F_i$, $i=1,2,3$, contains a matching saturating $N/2$ vertices,  each vertex set $V_i$, $i=1,2,3$, 
		can be partitioned into three sets, $S_i$, $T_i$, and $U_i$, which fulfil properties (i)-(iv) of Lemma~\ref{tutte}. We will show that it leads to a contradiction. 
			
		In order to simplify the description of the argument we shall always assume that none of the sets $T_i$ contains edges of the $i$th colour. 
		If this is the case we can always remove them; thus, to make our assumption valid we shall prove that the assertion holds not only for all $b$-dense graphs  $G$, 
		but for all  $b'$-dense graphs $G$, where $b'=b+3\sqrt N<N/8-7\sqrt N$. Moreover, 
from now on we set $A_{123}=V_1\cap V_2 \cap V_3$, $A_{ij}=(V_i\cap V_j)\setminus A_{123}$ for $1\le i<j\le 3$, and $$A_i=V_i\setminus (V_j\cup V_k)=V_i\setminus (A_{ij}\cup A_{ik}\cup A_{ijk})$$ 
		where here and  below  the indices $i,j,k$ are always 
		assumed to be  different elements of the set $\{1,2,3\}$.  Note that for such $i$, $j$, $k$, we have  
		$$V_i\setminus V_j = A_i\cup A_{ik}. $$

		We start with a list of properties of the above defined sets which follow
		directly from the definition and our assumptions on $G$ and components $F_1$, $F_2$ and $F_3$.
		
		\begin{claim} \label{cl:zero}
			\begin{enumerate}
				\item[(i)] All edges between $A_i$ and $A_{123}$ are of the $i$th colour.
				\item [(ii)] All edges between $A_{ij}$ and $A_{ik}$ are of the $i$th colour.
				\item[(iii)] $|V|-|V_i|= |A_j|+|A_k|+|A_{jk}|> N/4$.
				\item[(iv)] $A_i\neq \emptyset$.
				\item[(v)] There are no edges between $A_i$ and $A_{jk}$. In particular, 
				$|A_{jk}|\le b'<N/8-7\sqrt N$.
				\item[(vi)] If $A_{ik}=\emptyset$ then at least 
				one of the sets $A_{ij}$ and $A_{jk}$ is empty as well.
				\item[(vii)] At most one of the sets $A_1$, $A_2$, $A_3$ is larger than $N/4$. 
			\end{enumerate}	
		\end{claim}	
		
		\begin{proof}
			The items (i)-(iii) follows from the definition of the sets  $A$'s and the assumption that all $|V_i|\le 3N/4$. To see (iv) observe that the fact that $A_i=\emptyset$ means that the vertex set $V$ are covered by two components and so, by Lemma~\ref{l:match2}, one of them must have at least $3N/4$ vertices contradicting our assumption. The first part of 
			(v) follows the definition of   $A_i$ and $A_{jk}$, while the second one is a consequence of (iv) and the fact that $G$ is $b'$-dense. 
			If $A_{ik}$ is empty then, by (iii), one of the sets $A_i$ and $A_k$, say, $A_i$, is larger than $N/8>b'$. Hence, by (v), $A_{jk}=\emptyset$. Finally, 
			(vi) follows from Lemma~\ref{l:match1}.
		\end{proof}
		
		The remaining part of the proof of Lemma~\ref{l:small} we split  into two cases. The idea of the argument in both of them are similar. 
		We define sets $T'_i=T_i\setminus (S_1\cup S_2\cup S_3)$, for $i=1,2,3$, and   
		show that are localized in a small part of $V$ (e.g $A_i\cap T'_1=\emptyset$), 
		and they are large enough to have a substantial  intersection, which, in turn, leads to a contradiction with the assumption that $G$ is $b'$-dense. However, although both  proofs follow a similar line, the details are somewhat different. 
		
		\medskip
		
		{\sl Case 1.} Each of the sets  $A_{12}$, $A_{13}$, $A_{23}$ is non-empty.    
		
		\smallskip		
		
		Note that from Claim~\ref{cl:zero}(v) it follows that each of the six sets 
		$A_1$, $A_2$, $A_3$, $A_{12}$, $A_{13}$, $A_{23}$, has at most $b'<N/8$ vertices. 
		From Claim~\ref{cl:zero}(iii),
		we infer that $|A_{123}|\ge N-6b'>N/4$ and $|V_i|\ge 5N/8$ for $i=1,2,3$.
		
		\begin{claim}\label{cla4}
			$T_i\cap A_i=\emptyset$.
		\end{claim}

		\begin{proof}
			Let us assume that $v\in A_1\cap T_1$. Since vertices of $((T_1\cup U_1)\cap A_{123})\cup A_{23}$ are non-neighbours of $v$, we have
			\begin{equation*}
			|(T_1\cup U_1)\cap A_{123}|+|A_{23}|<N/8,
			\end{equation*}
			and 
			\begin{equation*}
			|S_1|\ge |S_1\cap A_{123}|> |A_{123}|-(N/8-|A_{23}|).
			\end{equation*}
			Hence
			\begin{equation}\label{eqq1}
			\begin{aligned}
			|T_1|&\ge |V_1|+|S_1|-N/2\\
			& > N-|A_2|-|A_3|-|A_{23}|+|A_{123}|-N/8+|A_{23}|-N/2\\
			&=3N/8+|A_{123}|-|A_2|-|A_3|. 
			\end{aligned}
			\end{equation}
			Consequently, since $|T_1\cap A_{123}|\le N/8-|A_{23}|$,
			\begin{equation*}\label{eqq2}
			|T_1\setminus A_{123}|\ge 3N/8+|A_{123}|-|A_2|-|A_3|+|A_{23}|-N/8>N/4. 
			\end{equation*}
			Thus, $T_1$ must have a non-empty intersection with all three sets, $A_1$, $A_{12}$, 
			and $A_{13}$.
			Since vertices of  $ (T_1\cup U_1)\cap A_{13}$ and vertices of $  A_3$ are non-neighbours of a vertex  $w \in T_1\cap A_{12}$, we have
			$$|(T_1\cup U_1)\cap A_{13}|+|A_3|<N/8.$$
			Thus
			\begin{equation}\label{eqq3}
			|S_1\cap A_{13}|\ge |A_{13}|-(N/8-|A_3|), 
			\end{equation}
			and similarly
			\begin{equation}\label{eqq4}
			|S_1\cap A_{12}|\ge |A_{12}|-(N/8-|A_2|). 
			\end{equation}
			But in this way we get a better lower bound for $S_1$ in  (\ref{eqq1}) and can replace it by  
			\begin{align*}\label{eqq5}
			|T_1\setminus A_{123}|&\ge N/8+ |A_{12}|+|A_{13}|> |V_1\setminus A_{123}|. 
			\end{align*}
			This contradiction shows that $T_1\cap A_1=\emptyset$. In the same way one can argue that $T_2\cap A_2=T_3\cap A_3=\emptyset$.
		\end{proof}   
		\begin{claim}\label{cla5}
			$U_i\cap A_i=\emptyset$.
		\end{claim}   
		
		\begin{proof}
			Let us assume again that, say, $v\in A_1\cap U_1$. Since vertices of $(T_1\cap A_{123})\cup A_{23}$  are  non-neighbours of $v$, we have 
			\begin{equation*}\label{eqq6}
			|T_1\cap A_{123}|< N/8-|A_{23}|,
			\end{equation*}
			and using the fact that   $|A_2|+|A_3|<N/4$, we obtain 
			\begin{equation}\label{eqq7}
		\begin{aligned}
		|T_1\setminus  A_{123}|&>|V_1|+|S_1|-N/2- N/8+|A_{23}|\\
		&\ge |V_1|-5N/8+|A_{23}| \ge 3N/8 -|A_2|-|A_3|\\
		& >N/8.
		\end{aligned}
			\end{equation}			
			Hence,  by  Claim~\ref{cla4}, $T_1$ must intersect both sets $A_{12}$ and $A_{13}$.
			Consequently (\ref{eqq3}) and (\ref{eqq4}) hold.
			Thus, using  (\ref{eqq3}) and  (\ref{eqq4}), we can improve  (\ref{eqq7}) to 
			\begin{align*}
			|T_1\setminus A_{123}|>|A_{12}|+|A_{13}|=|V_1\setminus (A_{123}\cup A_1)|. 
			\end{align*}
			This contradiction completes the proof. 
		\end{proof} 		
		
		
		
				Now, for $i=1,2,3$, we set $T'_i=T_i\setminus (S_1\cup S_2\cup S_3)$.  
		Note that from Lemma~\ref{tutte}(iii), for $i=1,2,3$, we get 
		\begin{align*}
		|T_i|-|S_i\setminus A_i|&\ge |V_i|+|A_i|-n\\
		&=N-n+|A_i|-|A_j|-|A_k|-|A_{jk}|.
		\end{align*}
		Since $T_i\cap A_i=\emptyset$, if we sum up the above inequalities for $i=1,2,3$, we arrive at
		\begin{equation}\label{eqq12}
		\begin{aligned}
	|T'_1|+|T'_2|+|T'_3|&\ge 2N-3n+|A_{123}|>3N/4-n/2+|A_{123}|\\
	&\ge N-n/2> 3N/4. 
		\end{aligned}
		\end{equation} 
		
		Note that $|V_i|< 3N/4$ and so, by Lemma~\ref{tutte}(iv)
		$$|T'_i|\le |T_i|\le 3N/4-n/2,\quad \textrm{for\ }i=1,2,3.$$ 
		Thus, at least two of the sets $T'_1$, $T'_2$, $T'_3$
		are larger than $N/8$. Let us assume that these are $T'_1$ and~$T'_2$. 
		
		Note that 
		$(T'_1\cup T'_2)\cap T_3
		=\emptyset$.
		Indeed, a vertex $v\in T'_1\cap T_3$ cannot be adjacent
		to any of at least $N/8-1$  elements   $w\in T'_2\setminus \{v\}$. Thus,   
		$$T'_1\cup T'_2 \subseteq (A_{123}\cup A_{12}\cup A_{13}\cup A_{23})\setminus T_3,$$
		and 
		\begin{align}
		|T'_1\cap T'_2|&=|T'_1|+|T'_2|-|T'_1\cup T'_2|\label{eqq111a}\\
		&\ge |T'_1|+|T'_2|+|T_3|\nonumber\\
		&\quad \quad-(|A_{123}|+|A_{12}|+|A_{13}|+|A_{23}|)\nonumber \\
		&\ge N/2-|A_{12}|-|A_{13}|-|A_{23}|-|T'_3|+|T_3|\label{eqq111b}\\
		& \ge N/8>b'.\nonumber
		\end{align}
		If $T'_3\neq \emptyset$, then take $v\in T'_3$ and note that, as before, none of the pairs $\{v,w\}$, where 
		$w\in (T'_1\cap T'_2)\setminus \{v\}$, is in $G$, contradicting the fact that $G$ is $b'$-dense 
		in~$K_{N}$. 
		
		Finally, let us consider the case when $T'_3=\emptyset$. 
		
		\begin{claim}\label{cla8}
			If $T'_3=\emptyset$, then  $(T'_1\cup T'_2)\cap A_{12}=\emptyset$. 
			
			In particular, 
			$T'_1\cup T'_2\subseteq A_{123}\subseteq V_3$.
		\end{claim}
		
		\begin{proof}
			Since $|V_3|>5N/8$, we have $|T_3|>N/8$. 
			Thus, from (\ref{eqq111b}), we get $|T'_1\cap T'_2|\ge N/4$. 
Let us assume that 
			$v\in  (T'_1\cup T'_2)\cap A_{12}$. Since $T'_1\cap T'_2\subseteq A_{123}\cup A_{12}$   and 
			$|A_{12}|<N/8$, from (\ref{eqq111b}) we deduce that at least    
			$$|T'_1\cap T'_2\cap A_{123}|\ge N/4-|A_{12}|>N/8$$
			vertices are not adjacent to $v$ which contradicts the fact that $G$ is $b'$-dense.
		\end{proof}
		
		From Claim~\ref{cla8} we conclude that  $T'_1\cup T'_2\subseteq U_3$ so, by Lemma~\ref{tutte}(iii),
		$$|T'_1\cup T'_2|\le |U_3|<N/2.$$
		Thus, (\ref{eqq12}) and  (\ref{eqq111a}) give 
		$$|T'_1\cap T'_2|> N/2+|A_{123}|-N/2>|A_{123}|,$$
		which clearly contradicts the fact that, due to Claim~\ref{cla8}, 
		we have $T'_1\cap T'_2\subseteq A_{123}$. This completes the proof of this case.
		
	\medskip

{\sl Case 2.} At least one of the sets $A_{12}$, $A_{13}$, and $A_{23}$ is empty.

\smallskip

	The proof of this case follows closely the idea of the previous one, yet we have much less control on the size of non-empty parts  of the
partition of $V$; in particular, it can happen that the size of $V_3$ is smaller than $N/2$. 
In order to make the argument easier to follow, we  split it into a number of short claims.
Let us start  with the following observation.  

\begin{claim}
	If $|V_3|<N/2$, then $A_{13}=A_{23}=\emptyset$.  
\end{claim}

\begin{proof}
	Let us suppose that, say, $A_{23}\neq \emptyset$ and thus by Claim~\ref{cl:zero}(vi),  $A_{12}=\emptyset$. Then, by Claim~\ref{cl:zero}(v), $|A_1|\le b'<N/8$. But then 
	$$|A_2|=|V|-|V_3|-|A_1|>3N/8,$$ 
	and 
	$$|A_{123}|\ge |V_1|-|A_1|> 3N/8,$$
	so $|V_2|=|A_2|+|A_{123}|>3N/4$ contradicting our assumptions on the size of the largest monochromatic component in~$G$. 
\end{proof}

Thus, throughout this case, we shall always assume that $A_{13}=A_{23}=\emptyset$.  If $|V_3|<N/2$ this choice is forced by the above claim, if 
$|V_i|\ge N/2$, for $i=1,2,3$, from Claim~\ref{cl:zero}(vi) at least two of the sets $A_{12}$, $A_{13}$, $A_{23}$ are empty, so  without loss of generality
we can assume that these are $A_{12}$ and $A_{23}$. 
Moreover, to be consistent with the notation, 
if $|V_3|< N/2$ we put by definition $T_3=\emptyset$ (we shall not look for the matching in such a small $F_3$  anyhow).

We start with the bounds for the size of the sets $A$'s. Note that Lemma~\ref{l:match2} implies that  
\begin{equation}\label{eqII:1}
|A_{123}|+|A_{12}|= |V_1\cap V_2|> N/4.
\end{equation}

If $A_{12}\neq \emptyset$, then, by Claim~\ref{cl:zero}(v),
\begin{equation}\label{eq:3333}
	|A_3|,|A_{12}|<N/8.
\end{equation}

Finally, the size of $A_i$'s are estimated by the following result.

\begin{claim}\label{cl:31}
	$|A_1|,|A_2|, |A_3|\le N/4$.
\end{claim}
\begin{proof}
	Suppose that the assertion does not hold and, say, $|A_1|>N/4$. Note that 
	$$|A_2|=N-|V_1|-|A_3|\ge N/8$$ 
	and so, by Lemma~\ref{l:match1}(iii), all vertices of the sets $A_1$ and $A_2$ are contained in one component $F'_3=(V'_3,E'_3)$ of the third colour. 
Let $W=V\setminus (V'_3\cup A_3)$. Note that $A_{123}\subseteq W\subseteq A_{123}\cup A_{12}$ and all edges joining $W$ and $A_1$ are of the first colour. 
We argue that $|W|\ge N/4$. Indeed, if  $A_{12}=\emptyset$, then, from \eqref{eqII:1}, we get $|W|= |A_{123}|>N/4$. 
Thus, let us assume that  $A_{12}\neq\emptyset$ and $|A_3|<N/8$. We consider two cases of the size of $V'_3$. If  $|V'_3|\le 5N/8$, then 
$$|W|\ge |V|-|V'_3|-|A_3|\ge N/4\,.$$
On the other hand, if  $|V'_3|\ge 5N/8$, then, by our choice of $F_1=(V_1,E_1)$ and $F_2=(V_2,E_2)$, we must have 
$$|V_1|,|V_2|\ge |V'_3|\ge 5N/8,$$ so, by Lemma~\ref{l:match2}, 
$$|W|\ge |A_{123}|\ge |V_1\cap V_2|-|A_{12}|\ge 3N/8-N/8=N/4\,.$$
Thus, from Lemma~\ref{l:match1}(iii) it follows that $F_1$ contains a matching saturating at least $N/2$ vertices, contradicting our assumption. 
\end{proof}

Using above result we can slightly improve \eqref{eqII:1}.

\begin{claim}\label{cl:31a}
	$|A_{123}|>N/4.$
\end{claim}

\begin{proof}
	For $A_{12}=\emptyset$ the assertion follows from \eqref{eqII:1}. If $A_{12}\neq \emptyset$, then, by \eqref{eq:3333} and Claim~\ref{cl:31},
	 we get
$$|A_{123}|=N-|A_1|-|A_2|-|A_3|-|A_{12}|>N/4\,.\qed $$
\renewcommand{\qed}{} 
\end{proof}

Let us also note the following consequence of Claim~\ref{cl:31}.

\begin{claim}\label{cl:a12}
	If $A_{12}\neq \emptyset$, then $|V_1|, |V_2|> 5N/8$, and  
	$|T_1|,|T_2|> N/8$.
\end{claim}
\begin{proof}
	For $i=1,2$, from Claims~\ref{cl:zero}(iii) and~\ref{cl:31}, we get
	\begin{equation*}
	|V_i|=N-|A_{3-i}|-|A_3|>5N/8
	\end{equation*}
	and so the lower bounds  for $T_i$ follows from Lemma~\ref{tutte}. 
\end{proof}

Before we state our next observation, let us recall that if $|V_3|<N/2$ we set $S_3=T_3=\emptyset$, $U_3=V_3$.
Moreover, similarly as in the previous case, for $i=1,2,3$, we put $S'_i=S_i\cap A_{123}$,  $S''_i=S_i\cap A_{i}$,
$U'_i=U_i\cap A_{123}$, and  $U''_i=U_i\cap A_{i}$.		 

\begin{claim}\label{cl:33}
	$T_i\cap A_i=\emptyset$ for $i=1,2,3$.

In particular,  $T_i\cap A_{123}\neq\emptyset$ for $i=1,2,3$, unless 
$i=3$ and $|V_3|<N/2$. 	
\end{claim}

\begin{proof} 
	Let $v\in T_i\cap A_i$, for some $i=1,2$. 
	Then, $|S'_i|\ge |A_{123}|-b'>N/8$. Note also that, by Lemma~\ref{tutte}(iii), 
	$|S_i|<N/4\le |A_{123}|$, and so there exists a vertex $w\in (U_i\cup T_i)\cap A_{123}$ and, consequently,  $|T_i\cap A_i|<N/8$.   
	
	If $T_i\cap A_{123}\neq \emptyset $, then 
	$|S''_i|\ge |A_i|-b'$, and so 
	$$|S_i|\ge |V_i|-N/4-|A_{12}|>N/4,$$
	which contradicts Lemma~\ref{tutte}(iii). Hence, $T_i\cap A_{123}=\emptyset $. 
	
	Thus, using Claim~\ref{cl:zero}(iii), we arrive at
	\begin{equation*}\label{eq:II3}
	\begin{aligned}
2|S_i|+|U_i|&\ge 2(|S'_i|+|S''_i|)+|U'_i|+|U''_i|\\
&\ge |S'_i|+|V_i|-|A_{12}|-|T_i\cap A_i|\\
&= |S'_i|+ N-|A_{3-i}|-|A_3|-|A_{12}|-|T_i\cap A_i|\\
&\ge 3N/4-|A_{3}|-|A_{12}|.
\end{aligned}
	\end{equation*}
Note that for $A_{12}\neq\emptyset$ from \eqref{eq:3333} we get $|A_3|+|A_{12}|\le  N/4$, while for 
$A_{12}=\emptyset$,	Claim~\ref{cl:31} gives $|A_3|+|A_{12}|=|A_3|\le  N/4$ as well. Consequently, 
$$2|S_i|+|U_i|>N/2$$
	which contradicts Lemma~\ref{tutte}, and so $A_i\cap T_i=\emptyset$ for $i=1,2$. 
	
If $|V_3|<N/2$, then $T_3=\emptyset$ and there is nothing to prove. 
	If $|V_3|\ge N/2$ and $w\in T_3\cap A_3$, then 
	$$|S_3|\ge |A_{123}|-b'=|V_3|-|A_3|-b'>N/4,$$
	contradicting  Lemma~\ref{tutte}(iii). Thus, $T_3\cap A_{3}=\emptyset $ as well. 
	
The second part of the assertion is an immediate consequence of Claim~\ref{cl:a12} and \eqref{eq:3333}.
\end{proof}

\begin{claim}\label{cl:35} 
	$T_1\cap T_2\cap A_{12}= \emptyset$.  
\end{claim}

\begin{proof}
	Suppose that  $v\in T_1\cap T_2\cap A_{12}$. Since there are no edges between $v$ and  $(A_{123}\cup A_3)\setminus (S_1\cup S_2)$, we have 
	$$|S'_1|+|S'_2|\ge |A_{123}|+|A_3|-b'.$$ 
Furthermore, from Claim~\ref{cl:a12}, for $i=1,2$, we have $|S''_i|\ge |A_i|-b'$, so
	$$|S_1|+|S_2|> |A_{123}|+|A_1|+|A_2|+|A_3|-3b'> 5N/8-|A_{12}|>N/2,$$ 
	which contradicts Lemma~\ref{tutte}(iii).
\end{proof}

\begin{claim}\label{cl:35a} 
$|U''_i|< N/8$ for $i=1,2,3$, unless $i=3$, $A_{12}=\emptyset$, and $|V_3|<N/2$.
\end{claim}

\begin{proof}
Let us first assume that $A_{12}\neq\emptyset$. Then $U''_3\subseteq A_3$ and so $|U''_3|<N/8$ by \eqref{eq:3333}. Moreover, for $i=1,2$,  
 Claims~\ref{cl:zero}(iii) and~\ref{cl:31} together with \eqref{eq:3333},  imply that
$|V_i|\ge 5N/8$. Consequently, by Lemma~\ref{tutte}, $|T_i|> N/8$ and so, by \eqref{eq:3333},  $T_i \cap A_{123}\neq \emptyset$ which results in $|U''_i|<N/8$.

Now let $A_{12}=\emptyset$. Then, by Claim~\ref{cl:33}, $T_i\subseteq A_{123}$ and the assertion follows unless $T_i=\emptyset$ which can happen only when  
$i=3$ and $|V_3|<N/2$. 
\end{proof}

Let  us assume that $|T_1|\ge |T_2|$. We shall argue that $|T_1\cap A_{123}| \ge N/8$ and so, in fact,  $U''_1=\emptyset$.
Indeed,  (\ref{eq:tt2}) implies that  for $i=1,2$, we have
\begin{equation}\label{eq:tt2}
\begin{aligned}
|T_i|&> |A_{123}\cup A_i|+|A_{12}|+|S_i|-N/2\\
&=|A_{123}|+2|A_i|+|A_{12}|-|U''_i|+|S'_{i}|-N/2\;.
\end{aligned}
\end{equation}
Hence, 
\begin{align*}
|T_1|+|T_2|&> 2|A_{123}|+2|A_1|+2|A_2|+2|A_{12}|-|U''_1|-|U''_2|-N\\
&\ge N-2|A_3|-|U''_1|-|U''_2|.
\end{align*}
Thus, if  $A_{12}= \emptyset$, then, from Claims~\ref{cl:31} and~\ref{cl:35a}, we have $|T_1|+|T_2|> N/4$ and 
$$|T_1\cap A_{123}|=|T_1|>N/8\,.$$ 
On the other hand, for  $A_{12}\neq \emptyset$  we get $|A_3|<N/8$ and $|T_1|+|T_2|> N/2$, which, in turn,
imply that $|T_1|\ge N/4$ and 
$$|T_1\cap A_{123}|=|T_1|-|A_{12}|>N/8\,.$$  
Consequently, $|T_1\cap A_{123}| > N/8$ and $U''_1=\emptyset$.

Using the above fact we argue that, under our assumptions,  the ``exceptional case'' mentioned in Claim~\ref{cl:35a} never occurs.

\begin{claim}\label{cl:35aa} 
	If $A_{12}=\emptyset$, then $|V_3|> N/2$.
\end{claim}

\begin{proof}
	Let $A_{12}=\emptyset$. Then, by Claim~\ref{cl:33}, $T_1\subseteq A_{123}$. Moreover, $U''_1=\emptyset$,  and 
so  Lemma~\ref{tutte} gives
	\begin{align*}
	|A_{123}|&\ge |T_1|> |V_1|+|A_1|-N/2=N/2-|A_2|-|A_3|+ |A_1|\\
	&=N/2+|A_1|-|A_2|-(|V_3|-|A_{123}|)\\
	&\ge N/2-|V_3|+ |A_{123}|,
	\end{align*}
Hence, $|V_3|> N/2$. 	
\end{proof}

Now, for $i=1,2,3$,  let $T'_i=T_i\setminus (S'_{1}\cup S'_{2}\cup S'_{3})$.
Note that 
\begin{equation}\label{eq:tt3}
|T_3|\ge  |A_{123}\cup A_3|+|S_3|-N/2=|A_{123}|+2|A_3|-|U''_3|+|S'_{3}|-N/2\;,
\end{equation} 
so from  (\ref{eq:tt2}), (\ref{eq:tt3}), and Claim~\ref{cl:35a} we get
\begin{equation}\label{eq:t3}
	\begin{aligned}
		|T'_1|+|T'_2|&+|T'_3|-|A_{123}|\\
		&\ge  |T_1|+|T_2|+|T_3|-|A_{123}|-|S'_{1}|-|S'_{2}|-|S'_{3}|\\
		&> 2(|A_{123}|+|A_1|+|A_2|+|A_3|+|A_{12}|)\\
		&\quad\quad\quad -|U''_2|-|U''_3|-3N/2\\
		&\ge N/2 -|U''_2|-|U''_3| > N/4.
	\end{aligned}
\end{equation}

\begin{claim}\label{cl:36} 
$|T'_1|,|T'_2|> N/8$ and 	$(T'_1\cup T'_2)\cap T'_3 = \emptyset$.  
\end{claim}

\begin{proof}
By Claim~\ref{cl:33}, we have $T_1,T_2, T_3\subseteq A_{12}\cup A_{123}$. Hence,  for $i=1,2$, we have 
$$|T'_i|+|T'_3|\le 	|T_i|+|T_3|\le |A_{123}|+|A_{12}|< |A_{123}|+N/8\,,$$
and so  \eqref{eq:t3} implies that $|T'_{3-i}|>N/8$. 	

Now	suppose that  $v\in T'_1\cap T'_3$. Note that there are no edges between $v$ and $T'_2$, and 
	$|T'_2|-1>N/8-1>b'$ which contradicts the assumption that $G$ is $b'$-dense. 
	In a similar way one can show that $T'_2\cap T'_3=\emptyset$. 
\end{proof}	

From the above fact we get 
$T'_1\cup T'_2\subseteq (A_{123}\setminus T'_3)\cup A_{12}$
so, using   (\ref{eq:t3}), we infer that
\begin{align*}\label{eq:t4a}
|T'_1\cap T'_2|&=|T'_1|+|T'_2|-|T'_1\cup T'_2|\\
&\ge |T'_1|+|T'_2|-|A_{123}\setminus T'_3|-|A_{12}|\ge N/8\,.\nonumber
\end{align*} 
Thus, by Claim~\ref{cl:35}, 
$$|T'_2\cap A_{123}|\ge |T'_1\cap T'_2\cap A_{123}|= |T'_1\cap T'_2| >N/8\,,$$ 
and so 
$U''_2=\emptyset$. 

But this means that  \eqref{eq:t3} can be improved to 
\begin{equation*}\label{eq:t33}
	|T'_1|+|T'_2|+|T'_3|-|A_{123}| \ge N/2-|U''_3|>3N/8.
\end{equation*}
Since $T'_1,T'_2\subseteq (A_{123}\setminus T'_3)\cup A_{12}$, from the above inequality and Claim~\ref{cl:36} we get 	
\begin{equation}\label{eq:t4}
	\begin{aligned}
		|T'_1\cap T'_2|&=|T'_1|+|T'_2|-|T'_1\cup T'_2|\\
		&\ge |T'_1|+|T'_2|-|A_{123}|-|A_{12}|+|T'_3|\\
		&> N/2-|A_{12}|-|U''_3|\ge N/4.
	\end{aligned} 
\end{equation}
and so, by Claim~\ref{cl:35}, 
\begin{equation}\label{eq:t41}
	|T'_1\cap T'_2\cap A_{123}|\ge N/4.
\end{equation} 

Note that if $T'_3\neq\emptyset$, then $v\in T'_3$ is not adjacent to any vertex 
from $T'_1\cap T'_2\cap A_{123}$ which contradicts the fact that $G$ is $b'$-dense. 
Thus, we may assume that  $T'_3=\emptyset$.

\begin{claim}\label{cl:37}
	$(T'_1\cup T'_2)\cap A_{12}=\emptyset$, so
	$T'_1\cup T'_2\subseteq A_{123}\subseteq V_3$.
\end{claim} 

\begin{proof}
	Let 
	$v\in  (T'_1\cup T'_2)\cap A_{12}$. From (\ref{eq:t41}) at least    
	$ N/4$ vertices are not adjacent to $v$ which contradicts the fact that $G$ is $b'$-dense.
\end{proof}

Note that by Claims~\ref{cl:33} and~\ref{cl:35aa} we have $T_3\cap A_{123}\neq \emptyset$ and, consequently,  $|U''_3|\le \min\{b', |A_3|\}$. 
Thus,  using Claim~\ref{cl:37}, from \eqref{eq:t4} we get 
$$|T'_1\cap T'_2|> N/2-\min\{b',|A_3|\}.$$
Therefore,  the set 
$W'=T'_1\cap T'_2\subseteq V_3$ has more than $N/2-|A_3|\ge 3N/8$ vertices, and
all edges contained in $W'$, as well as  
 all edges joining $W'$ to $A_3$, are of the third colour. 
Thus, 
we can use Lemma~\ref{l:match1}(ii) to match  $\min\{|A_3|,N/8\}$ vertices of $A_3$ to vertices of $W'$ and apply Dirac's Theorem to saturate
yet unsaturated vertices of $W'$ by a matching disjoint from the previous one.
\end{proof}

\begin{proof}[Proof of Theorem~\ref{thm:main}]Theorem~\ref{thm:main} is an immediate consequence of 
	Lemmata~\ref{l:giant} and~\ref{l:small}.
\end{proof}

\section{Example}\label{sec:example}

Finally, we give an example of a graph $\hG=(\hV, \hE)$ on $8\ell$ vertices with $\delta(\hG)=7\ell-2$
for which the
property $\hG\to (C_{4\ell},C_{4\ell},C_{4\ell})$ does not hold. 

Let $\hV=\{a,b\}\cup A_1 \cup A_2 \cup A_3 \cup A_4 $, where 
$|A_1|=|A_4|=2\ell-2$ and $|A_2|=|A_3|=2\ell+1$. 
First we  put in $\hG$ all edges inside the set $A_1\cup A_3$, inside the set $A_2\cup A_4$, 
and the edge $\{a,b\}$, and colour them red. Clearly, in this way we create no red components of size $4\ell$ and thus no red cycles of length $4\ell$.

Next we add to $\hG$ all edges joining $a$ to $A_1\cup A_2\cup A_3\cup A_4$, all edges between 
sets $A_1$ and $A_2$, and all edges between $A_3$ and $A_4$. All these edges we colour blue. Note that $a$ is a cut vertex of the blue graph, so each blue cycle is contained either in $\{a\}\cup A_1\cup A_2$, or in $\{a\}\cup A_3\cup A_4$. But each of these two graphs is isomorphic to $K_{1,2\ell-2,2\ell+1}$, and so 
has $4\ell$ vertices and independence number $2\ell+1$ which, clearly, implies that it contains no  $C_{4\ell}$.

To define the green colouring we need to partition each set $A_i$ into two sets $A_{i1}\cup A_{i2}$, $1\le i\le 4$,  with almost equal sizes, i. e., 
$|A_{11}|=|A_{12}|=|A_{41}|=|A_{42}|= \ell-1$, $|A_{21}|=|A_{31}|= \ell$, and
$|A_{22}|=|A_{32}|= \ell+1$. Now we add to $\hG$ all edges joining $b$ to  
$A_1\cup A_2\cup A_3\cup A_4$ and all edges between four pairs of sets:  $(A_{11},A_{42})$,  $(A_{12},A_{41})$, $(A_{21},A_{31})$, and $(A_{22},A_{32})$. We colour all these edges by green.
Then, after removing the vertex $b$, the green graph decomposes into four components each of them 
of at most $2\ell+2$ vertices. Thus, there are no green cycles of length $4\ell$.

Finally, observe that the only pairs which are not edges of $\hG$ are those between 
sets  $(A_{11},A_{41})$,  $(A_{12},A_{42})$, $(A_{21},A_{32})$, and $(A_{22},A_{31})$.
Thus, $\delta(\hG)=7\ell-2$.	

\usetikzlibrary{decorations.pathmorphing} 	

\begin{figure}
	\begin{tikzpicture}[ decoration={bent,aspect=.3}] 
	\draw (4,-6) circle (0.04cm) ;
	\put (110,-160) {\large $a$}
	\draw (17,-6) circle (0.04cm);
	\put (463,-160) {\large $b$}
	\draw [draw=red, very thick, loosely dashed]  (4,-6) -- (17,-6);
	\node[circle,draw] (A) at (9,-3) {$A_{11}$}; 
	\node[circle,draw] (B) at (7,-4) {$A_{12}$};
	\node[circle,draw] (C) at (14,-3) {$A_{21}$};
	\node[circle,draw] (E) at (9,-8) {$A_{31}$};
	\node[circle,draw] (D) at (12,-4) {$A_{22}$};
	\node[circle,draw] (F) at (7,-9) {$A_{32}$};
	\node[circle,draw] (G) at (14,-8) {$A_{41}$};
	\node[circle,draw] (H) at (12,-9) {$A_{42}$};
	\draw [draw=blue, very thick, dotted]  (4,-6) -- (C);
	\draw [draw=blue, very thick, dotted]  (4,-6) -- (A);
	\draw [draw=green, very thick]  (17,-6) -- (A);
	\draw [draw=blue,  very thick, dotted]  (4,-6) -- (B);
	\draw [draw=blue,  very thick, dotted]  (4,-6) -- (G);
	\draw [draw=blue,  very thick, dotted]  (4,-6) -- (H);
	\draw [draw=blue,  very thick, dotted]  (4,-6) -- (D);
	\draw [draw=blue,  very thick, dotted]  (4,-6) -- (E);
	\draw [draw=blue,  very thick, dotted]  (4,-6) -- (F);
	\draw[draw=green, very thick] (17,-6) -- (C); 
	\draw[draw=green, very thick] (17,-6) -- (D); 
	\draw[draw=green, very thick] (17,-6) -- (E); 
	\draw[draw=green, very thick] (17,-6) -- (F); 
	\draw [draw=green, very thick]  (17,-6) -- (A);
	\draw [draw=green, very thick]  (17,-6) -- (B);
	\draw [draw=green,very  thick]  (17,-6) -- (G);
	\draw [draw=green, very thick]  (17,-6) -- (H);
	\draw[draw=red, loosely dashed, very thick] (E) -- (F); 
	\draw[draw=green,very thick] (A) -- (H);  
	\draw[draw=green,very thick] (B) -- (G); 
	\draw[draw=green,very thick] (D) -- (E); 
	\draw[draw=red,loosely dashed, very thick] (A) -- (B); 
	\draw[draw=blue,very thick, dotted] (A) -- (C); 
	\draw[draw=blue,very thick, dotted] (A) -- (D); 
	\draw[draw=blue,very thick, dotted] (B) -- (C); 
	\draw[draw=blue,very thick, dotted] (B) -- (D); 
	\draw[draw=red,loosely dashed, very thick] (C) -- (D); 
	\draw[draw=red,loosely dashed, very thick] (C) -- (G);  
	\draw[draw=red,loosely dashed, very thick] (C) -- (H); 
	\draw[draw=red,loosely dashed, very thick] (D) -- (G); 
	\draw[draw=red,loosely dashed, very thick] (D) -- (H);     
	\draw[draw=red,loosely dashed, very thick] (E) -- (A); 
	\draw[draw=red,loosely dashed, very thick] (E) -- (B); 
	\draw[draw=red,loosely dashed, very thick] (F) -- (A); 
	\draw[draw=red,loosely dashed, very thick] (F) -- (B); 
	\draw[draw=red,loosely dashed, very thick] (G) -- (H);  
	\draw[draw=blue,very thick,dotted] (G) -- (E);  
	\draw[draw=blue,very thick, dotted] (G) -- (F);  
	\draw[draw=blue,very thick,dotted] (H) -- (E);  
	\draw[draw=blue,very thick,dotted] (H) -- (F);  
	\draw[draw=green,very thick] (C) -- (F);  
	\node[circle,draw,fill=red!30] at (9,-3) {$A_{11}$}; 
	\node[circle,draw,fill=red!30]  at (7,-4) {$A_{12}$};
	\node[circle,draw,fill=red!30]  at (14,-3) {$A_{21}$};
	\node[circle,draw,fill=red!30]  at (9,-8) {$A_{31}$};
	\node[circle,draw,fill=red!30]  at (12,-4) {$A_{22}$};
	\node[circle,draw,fill=red!30]  at (7,-9) {$A_{32}$};
	\node[circle,draw,fill=red!30]  at (14,-8) {$A_{41}$};
	\node[circle,draw,fill=red!30]  at (12,-9) {$A_{42}$};
	\end{tikzpicture}
	\caption{The colouring of the graph $\hat G$. The dotted, dashed, and solid lines stand for different colours of edges joining two sets of vertices.}
\end{figure}
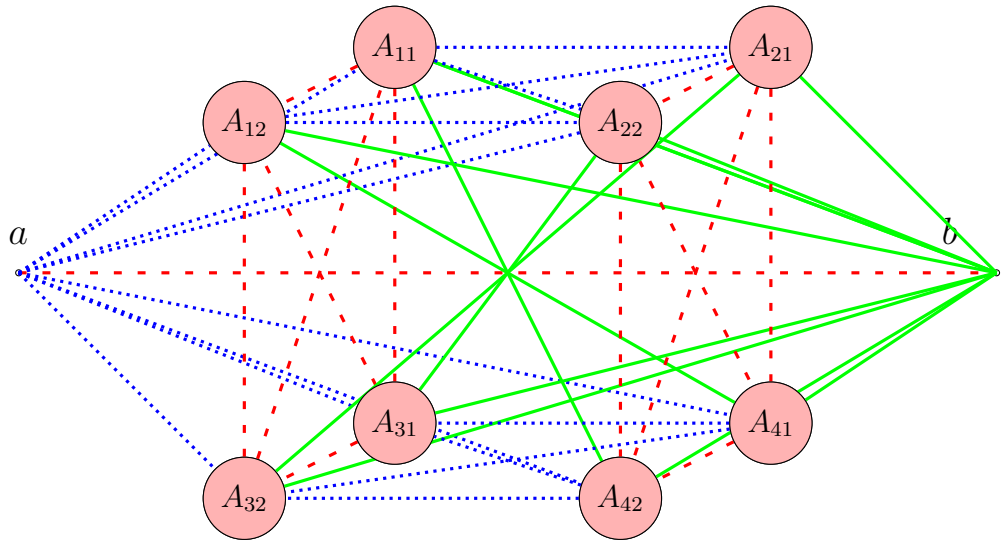

\break

\bibliographystyle{plain}

\end{document}